\documentclass[11pt]{article}
\usepackage{amssymb,amsmath}
\usepackage{amsthm,color}
\newtheorem{theorem}{Theorem}[section]
\newtheorem{proposition}{Proposition}[section]

\newtheorem{definition}{Definition}[section]

\DeclareMathOperator{\End}{End}

\DeclareMathOperator{\Id}{Id}
\DeclareMathOperator{\Ind}{Ind}

\begin{document}

\title{Conformal covariance for the powers of the Dirac operator }

\author{Jean-Louis Clerc and Bent \O rsted}

\date{ }
\maketitle

\begin{abstract}
A new proof of the conformal covariance of the powers of the flat Dirac operator is obtained. The proof uses their relation with the Knapp-Stein intertwining operators for the spinorial principal series.
We also treat the compact picture, i.e. the corresponding operators on the sphere, where certain
polynomials of the Dirac operator appear. This gives a new representation-theoretic framework for earlier
results in \cite{bo, es, lr}.  
\end{abstract}
\section{Introduction}

The powers of the (flat) Dirac operator are known to satisfy a covariance property with respect to the  M\"obius group (see \cite{b,pq, es}). We give a new proof by interpreting the powers of the Dirac operator as \emph{residues} of a meromorphic family of Knapp-Stein intertwining operators. The proof is elementary and does not involve any Clifford analysis.
We also give in the last section the corresponding meromorphic family and its residues on the sphere. This
corresponds to the so-called \emph{compact picture} of the induced representations.
In particular we find as residues on the sphere the polynomials $D(D^2 - 1^2) \cdots (D^2 - m^2)$ of the Dirac
operator $D$. These were found earlier by other methods in \cite{es, lr} (in these references the polynomials
are not with constant coefficients, and it still needs some consideration to see that they may be identified with
the polynomials above).  

\section{The conformal group of the sphere}

Let $\big(E^{n+1}, (\,.\,,\,.\,)\big)$ be a Euclidean vector space of dimension $n+1$, and denote by $(x,y)$ the inner product of two vectors $x$ and $y$. Let $S=S^n$ be the unit sphere of $E$. The \emph{Lorentz space} is
 \[ E^{1,n+1} =\mathbb R \oplus E= \{(t, x), t\in \mathbb R, x\in E^{n+1}\}\] 
 endowed with the symmetric bilinear form $[\,.\,,\,.\,]$ given by
 \[ [(t,x), (u,y)] =tu -(x,y)\ .\]
 
Let $\mathcal S$ be the set of isotropic lines in $E^{1,n+1}$. The map
\[ S\ni x\longmapsto d_x = \mathbb R(1,x) \in \mathcal S\]
is a one-to-one correspondance of $S$ onto $\mathcal S$, which is moreover a diffeomorphism for the canonical differential structures on $S$ (as a submanifold of $E^{n+1}$) and $\mathcal S$ (as a submanifold of the projective space of $E^{1,n+1}$).
 
 Let $G= SO_0(E^{1,n+1})\simeq SO_0(1,n+1)$ be the connected component of the neutral element in the group of isometries of $E^{1,n+1}$. Then $G$ acts on $\mathcal S$ and hence on $S$. This action is \emph{conformal} in the sense that for any $g$ in $G$ and $x\in S$, the differential $Dg(x) : T_x\longrightarrow T_{g(x)}$ is a similarity of the tangent space $T_x$ of $S$ at  $x$, i.e.
 \[ Dg(x) = \kappa(g,x)\, r(g,x)\ ,\]
 where $r(g,x)$ is a positive isometry of $T_x$ into $T_{g(x)}$ and  $\kappa(g,x) >0$ is the \emph {conformal factor} of $g$ at $x$.
  
 The group $K=SO(E^{n+1})\simeq SO(n+1)$ is a maximal compact subgroup of $G$, associated to the standard Cartan involution $g\longmapsto \theta(g) = (g^t)^{-1}$. The group $K$ acts transitively on $S$.
 
 Let $(e_0,\dots, e_n)$ be an orthonormal basis of $E^{n+1}$, and choose $e_0$ as origin in $S$. Let $E^n= e_0^\perp$ be the hyperplane orthogonal  to $e_0$ in of $E^{n+1}$. The stabilizer of $e_0$ in $K$ is the subgroup $M\simeq SO(n)$, and this gives a realization of $S\simeq K/M$ as a \emph{compact Riemannian symmetric space}. On the other hand, let $P$ be the stabilizer of $e_0$ in $G$. Then $P$ is a parabolic subgroup of $G$, and $S\simeq G/P$ is a realization of $S$ as a \emph{flag manifold}.
 
 Let $e_{-1}=(1,0, \dots, 0)\in E^{1,n+1}$, so that $\{e_{-1},e_0,\dots, e_n\}$ is a basis of $E^{1,n+1}$. Introduce the subgroups of $G$ defined by
 
\[A = \left\{ a_s = \begin{pmatrix}\cosh s&\sinh s&0&\dots&0\\\sinh s& \cosh s&0&\dots&0\\0&0&1 & &\\ \vdots&\vdots&&\ddots&\\0&0&&&1\end{pmatrix},\quad s\in \mathbb R\right\}
\]
and
\[N= \left\{\ n_u =\begin{pmatrix}1+\frac{\vert u\vert^2}{2}&-\frac{\vert u\vert^2}{2}&&u^t&\\\frac{\vert u\vert^ 2}{2}&1-\frac{\vert u\vert^2}{2}&&u^t&\\
&&1&&
\\u&-u&&\ddots&\\&&&&1
\end{pmatrix},\quad u\in \mathbb R^n\simeq E^n\right\}\ .
\]
 Then $P=MAN$ is a Langlands decomposition of $P$. Let $\overline N = \theta N$ be the opposite unipotent subgroup, given by
  \[ \overline N =  \left\{\ \overline n_v =\begin{pmatrix}1+\frac{\vert v\vert^2}{2}&\frac{\vert v\vert^2}{2}&&v^t&\\-\frac{\vert v \vert^ 2}{2}&1-\frac{\vert v\vert^2}{2}&&-v^t&\\
&&1&&
\\v&v&&\ddots&\\&&&&1
\end{pmatrix},\quad v\in \mathbb R^n\right\}\ .
 \]
 
 We also let \[w=\begin{pmatrix} 1& & & & &\\& -1& & & & &\\ & &-1&& & &\\ & & &1& & &\\ & & & & \ddots &\\ & & & & &  1\end{pmatrix}\ .\]
 The element $w$ is in $K$, acts on $S^n$ by $(x_0,x_1,x_2,\dots, x_n) \longmapsto (-x_0,-x_1,x_2,\dots, x_n)$, satisfies $wa_sw^{-1}=a_{-s}$, thus realizes the non trivial Weyl group element. 
  The map 
 \[c : v \longmapsto \Big(\frac{\vert v\vert^2-1}{\vert v \vert^2+1},\  \frac{2}{\vert v\vert^2+1}\,v  \Big)
 \] is a diffeomorphism from $E^n$ onto  $S\setminus\{e_0\} $. Its inverse is  the classical  \emph{stereographic projection} from the source ${e_0}$. 
 
 As a convention let $\overline E^n = E^n\cup \infty$ be the one-point compactification of $E_n$. Then clearly the map $c$ can be extended to $\overline E_n$  by setting $c(\infty) = e_0$, to get a one-to-one correspondance between $\overline E^n$ and $S^n$. This allows to transfer the action of $G$ on $S$ to an action of $G$ on $\overline E^n$. In this way, the group $G$ is realized as a group of rational conformal transformations of  $E^n$, usually called the \emph{M\"obius group} $M_+(\overline E_n)$. The subgroup $P$ is now realized by affine similarities, the group $M$ is the group of rotations of $E$ with center at $0$, $A$ is the group of positive dilations with center $0$ and $N$ is the group of translations of $E$. The element $w$ acts as the \emph{twisted inversion}
 \[(x_1,x_2,\dots, x_n) \longmapsto \Big(-\frac{x_1}{\vert x\vert^2}, \frac{x_2}{\vert x\vert^2}, \dots, \frac{x_n}{\vert x\vert^2}\Big)\ .
 \] 
  When using this model for the sphere, we refer to the \emph{noncompact picture}.
  
 \section{The Vahlen-Maass-Ahlfors realization of the twofold covering of the M\"obius group}

There is a quite useful realization of (a twofold covering of) $G$ acting on $\overline E_n$ via the Clifford algebra $Cl_{n-1}$, initiated by Vahlen (\cite{v}), revived by Maass (\cite{m}) and well presented by Ahlfors (\cite{a}, see also \cite{gm}, \cite{r1}).

Let $E^{n-1}$ be the vector subspace generated by $e_2,\dots, e_n$, and form the  \emph{Clifford algebra} $Cl_{n-1} = Cl(E^{n-1})$, i.e. the algebra (with unit $1$) generated by the vector space $E^{n-1}$ and the relations
\[uv+vu +2(u,v) = 0\ .
\]
The space $E^n$ is identified with the subspace $\mathbb R \oplus E^{n-1}$ of $Cl(E^{n-1})$, the element $e_1$ corresponding to the unit of the Clifford algebra. Following Ahlfors, elements of $E^n$ will be called \emph{vectors}.

Recall the three involutions of the Clifford algebra : the \emph{principal automorphism} $a\mapsto a'$, obtained by sending $e_j$ to $-e_j$ for $2\leq j\leq n$, the \emph{reversion} $a\mapsto a^*$ which is the antiisomorphism mapping $e_j$ to $e_j$ for $2\leq j \leq n$ and the \emph{conjugation} $a\mapsto \overline{a}$ which is the composition of the two first. 

There exists a canonical inner product on $Cl_{n-1}$ extending the inner product on $E^{n-1}$, the corresponding norm is denoted by $\vert \ \vert$.

For vectors $x=x^*$, $x'=\overline x$, and $x\overline x = x_1^2+\dots + x_n^2=\vert x\vert ^2$, so that any vector $x\neq 0$ is invertible with inverse equal to $x^{-1} = \frac{\overline x}{\vert x\vert^2}$. 

The \emph{Clifford group} $\Gamma_n$ is the set of all elements of the Clifford algebra $Cl_{n-1}$ which can be written as products on non-zero vectors.
If $a, b $ are in  $\Gamma_n$, then \[ a\overline a = \vert a\vert^2\ ,\quad \vert ab\vert = \vert a\vert \vert b\vert \]

Let $a$ be in $\Gamma_n$. Then $a$ is invertible and  $a^{-1} = \frac{\overline a}{\vert a\vert^2}$. For $x$ in $E^n$, $ax{a'}^{-1}\in E^n$, and the map $x\mapsto ax{a'}^{-1}$ is a positive isometry of $E^n$. 

\begin{definition}\label{Cmatrix}
 A matrix $g=\begin{pmatrix} a&b\\ c&d\end{pmatrix}$ is a \emph{Clifford matrix} if

$i) \ a,b,c,d \in \Gamma_n \cup \{0\}$

$ii)\ ad^* - bc^* = 1$

$iii)\ ab^* \text{ and } cd^* \in E^n$

\end{definition}

\begin{theorem} Under matrix multiplication, the Clifford matrices form a group, denoted by $SL_2(\Gamma_n)$. 

\end{theorem}

The following observation will be useful later : for $a, b\in \Gamma_n$, the conditions $ab^*\in E^n$ and $a^{-1}b \in E^n$ are equivalent.

Let $g=\begin{pmatrix} a&b\\ c&d\end{pmatrix}$ in $SL_2(\Gamma_n)$ and $x$ in $\overline E^n$. We recall the meaning of the expression $g(x) = (ax+b)(cx+d)^{-1}$.

First $cx+d$ is either $0$ or is invertible :  

$\bullet$ if $c\neq 0$, $c$ is in $\Gamma_n$. Condition $iii)$ and the remark imply that $c^{-1}d\in E^n$, so that $cx+d = c(x+c^{-1}d)$ is either $0$ or invertible.

$\bullet$ if $c=0$, then $ii)$ implies that $d\neq 0$ so that $cx+d=d$ is in $\Gamma_n$, hence is invertible.

Next, observe that $ax+b$ and $cx+d$ cannot be both $0$. In fact assume the contrary. By $ii)$, $a$ and $c$ can not be both $0$. So, assume that $a\neq 0$. Hence $a$ is invertible and $x=-a^{-1} b= b^* {a^*}^{-1}$. Hence \[0 = cx+d = -c b^* (a^*)^{-1} +d\ .\]
Multiplying both sides on the left by $a^*$ yields $-c^*b+d a^*= 0$ which is incompatible with $ii)$. A similar argument holds if $c\neq 0$.

So, if $cx+d\neq 0$, $(ax+b)(cx+d)^{-1}$ is a well defined element of $Cl_{n-1}$. If $cx+d=0$, then set $(ax+b)(cx+d)^{-1} =\infty$. Finally, let $(a\infty +b)(c\infty +d)^{-1} = ac^{-1}$ is $c\neq 0$ and $=\infty$ if $c=0$. 

\begin{theorem} ${ }$

\medskip

$i)$ For any $g$ in $SL_2(\Gamma_n)$ and $x\in \overline E^n$, $g(x) = (ax+b)(cx+d)^{-1}$ belongs to $\overline E^n$ 

\medskip

$ii)$ The map
$\iota_g : \overline E^n \ni x\mapsto g(x) \in \overline E^n$ belongs to the M\"obius group $M(\overline E^n)$

\medskip

$iii)$ The homomorphism $g\mapsto \iota_g$ is a twofold covering of the M\"obius group, with kernel $\pm\Id_2$.
\end{theorem}
See \cite{a} for a proof.

In the sequel , we let $\widetilde G = SL_2(\Gamma_n)$. The stabilizer of $\infty$ in $\widetilde G$ is the subgroup \[\widetilde P=\Bigg\{ \begin{pmatrix} a&b\\0&d\end{pmatrix},\quad ad^* =1,\ ab^*\in E^n\Bigg\}\ .\]
The Langlands decomposition of $\widetilde P$ is $\widetilde P = \widetilde L N = \widetilde M AN$, where 

\[ N= \Bigg\{ \begin{pmatrix} 1&v\\0&1\end{pmatrix},\quad v\in E^n \Bigg\}\]

\[\widetilde L= \Bigg\{ \begin{pmatrix} a&0\\0&{a^*}^{-1}\end{pmatrix},\quad a\in \Gamma^n \Bigg\}\]
\[ A = \Bigg\{a_t= \begin{pmatrix} t&0\\0&t^{-1}\end{pmatrix},\quad  t>0\Bigg\}\]

and the group $\widetilde M$ is realized as
\[\widetilde M =\Bigg\{ \begin{pmatrix}m&0\\0&{m^*}^{-1}\end{pmatrix}, \quad m\in \Gamma_n, \vert m\vert =1 \Bigg\}\ .
\]

The element $\begin{pmatrix}m&0\\0&{m^*}^{-1}\end{pmatrix}$ of 
$\widetilde M$ acts on $E^n$ by $\sigma_m : x\longmapsto mxm^*$ and the map $\sigma : m\longmapsto \sigma_m$ is a twofold covering of $\widetilde M \simeq Spin_n$ onto $SO(n)$.

Let $H = \begin{pmatrix} 1&0\\0&-1\end{pmatrix}$, so that, for $t>0$ $a_t=\exp \log t H$. A complex linear form on $\mathfrak a$ is identified with the complex number $\lambda$ equal to the value of the linear from on the element $H$. The half-sum of the roots $\rho$ corresponds to the number $n$, and we let $a_t^\lambda= t^\lambda , t>0, \lambda\in \mathbb C$.

The group $N$ acts by translations $x\mapsto x+v$.

An element of $\overline N$ is of the form $\begin{pmatrix}1&0\\x&1 \end{pmatrix}$ where $x\in E^n$. We will identify $\overline N$ with $E_n$.

The non trivial Weyl group element  is realized by the matrix $ w= \begin{pmatrix} 0&-1\\ 1&0 \end{pmatrix}$.

The \emph{Bruhat decomposition} of an element $g=\begin{pmatrix} a&b\\c&d\end{pmatrix}$, where $a\neq 0$ is given by

\[\begin{pmatrix} a&b\\c&d\end{pmatrix} = \begin{pmatrix} 1&0\\ca^{-1}&1\end{pmatrix} \begin{pmatrix} a&0\\0&{a^*}^{-1}\end{pmatrix}\begin{pmatrix} 1&a^{-1}b\\0&1\end{pmatrix}\ .
\]
The proof of this identity reduces to \ $\ ca^{-1}b+{a^*}^{-1}= d\ $\  , which is a consequence of the assumptions on $a,b,c$ and $d$ (see Definition \ref{Cmatrix}).

\begin{proposition}
 Let $g\in \widetilde G$, $x,y \in E^n$ and assume that $g(x), g(y)\in E^n$. Then
\begin{equation}\label{globalcov}
g(x)-g(y) = {(cy+d)^*}^{-1}(x-y) (cx+d)^{-1}\ .
\end{equation}
\end{proposition}

\begin{proposition}\label{localcov}
 Let $g$ be in $\widetilde G$, $x\in E^n$ and assume that $g(x)\in E^n$. Then the differential of the action of $g$ at $x$ is given by
\begin{equation} Dg(x) \xi= {(cx+d)^*}^{-1}\xi (cx+d)^{-1}\ .\end{equation}
The conformal factor of $g$ at $x$ is given by $\kappa(g,x) = \vert cx+d\vert^{-2}$ and the rotation factor is given by $\sigma((cx+d)' \vert cx+d\vert^{-1})$.
\end{proposition}

Proposition \ref{globalcov} is proved in \cite{a} (see also \cite{gm}). The formula for the differential in Proposition \ref{localcov} is a consequence. For the last part of the proposition, observe that $a=(cx+d)$ is in $\Gamma_n$  and
\[{a^*}^{-1} \xi a^{-1} = \vert a\vert^{-2} \Big(\frac{a}{\vert a\vert}\Big)^{*-1}\,\xi\, \Big(\frac{a}{\vert a\vert}\Big)^{-1}= \vert a\vert^{-2}\,\frac{a'}{\vert a\vert}\,\xi\, \Big(\frac{a'}{\vert a\vert}\Big)^*\ .
\]
Another formula will be useful later.
\begin{proposition} Let $g=\begin{pmatrix} a&b\\c&d\end{pmatrix}$ be in $SL_2(\Gamma_n)$. Let $x\in E^n$, assume that $g$ is defined at $x$ and let $y=(ax+b)(cx+d)^{-1}$. Then
\begin {equation}\label{invcov}
(-c^*y+a^*) = (cx+d)^{-1}
\end{equation}

\end{proposition}

\begin{proof} The identity follows from
\[ (-c^*y+a^*)(cx+d)= -c^*(ax+b) +a^*(cx+d) = -c^*b+a^*d\ ,\]
and the fact that $a^*d -c^*b=1$, a consequence of the conditions satisfied by $a,b,c,d$ (see Definition\ref{Cmatrix}).
\end{proof}
The map $\widetilde \theta$ defined on $\widetilde G$ by
\[\widetilde \theta \ \begin{pmatrix} a&b\\c&d\end{pmatrix} = \begin{pmatrix} d'&-c'\\-b'&a'\end{pmatrix}
\]
is an involution of $\widetilde G$, which covers the standard involution $\theta$ on $G$. The fixed points set of $\widetilde \theta$ is the subgroup $\widetilde K$  given by
\[\widetilde K = \Bigg\{ \begin{pmatrix} a&b\\-b'&a'\end{pmatrix}, a,b,\in \Gamma_n\cup \{0\}, \vert a\vert^2+\vert b\vert^2 = 1, ab^*\in E^n\Bigg\}\ .
\]
The subgroup $\widetilde K$ is a maximal compact subgroup of $\widetilde G$, isomorphic to $Spin(n+1)$ and is a twofold covering of $K$.

The sphere $S^n$ can be interpreted as $\widetilde G/\widetilde P$ (flag manifold) and as $\widetilde K/\widetilde M$ (compact symmetric space).

To determine the Lie algebra of $Spin_n$ in this model, we will describe one-parameter groups, and find the associated vector field.

First consider, for $2\leq j\leq r$
\[ m_t = \cos \frac{t}{2} +\sin \frac{t}{2} e_j, \quad t\in \mathbb R\ .\]
Then \[\sigma_{m_t}= \begin{pmatrix}\cos t & & -\sin t& &\\&1&&\\ \sin t& &\cos t&&\\ & & &1&\\\end{pmatrix}\]

For $2\leq j<k$, let 
\[ m_t = \cos \frac{t}{2} e_j + \sin \frac{t}{2} e_k\ . \]

Then
\[ \sigma_{m_t} = \begin{pmatrix} 1&&&&&&&&\\&&\cos t&&&-\sin t&&\\&&&&1&&&&\\&&\sin t&&&\cos t&&\\&&&&&&&&1\end{pmatrix}
\]
So, a basis of the Lie algebra of the spin group is
\[\frac{1}{2} e_j,\quad  2\leq j \leq n,\quad \frac{1}{2} e_je_k,\quad  2\leq j<k\leq n \ .
\]

\section{The spinor representation} 

We recall some well-known results on the spinor representations. We will use the standard realization of $Spin$ in $ Cl_n$, (not to be confused with the realization of the same group in $Cl_{n-1}$ we used in the previous section), namely

\[Spin_n=\{ a=v_1\dots v_{2k},\quad v_j\in \mathbb R\oplus \mathbb R^n, \vert v_j\vert =1, k\in \mathbb N\}\ .
\]

A finite dimensional complex Hilbert space $\mathcal H$ is said to be a \emph{Clifford module} if there exists $n$ skew-Hermitian operators $E_1,\dots, E_n$ on $\mathcal H$, such that
\begin{equation}\label{Ej}
E_iE_j + E_jE_i=-2\delta_{ij} \Id, \qquad 1\leq i,j\leq n\ .
\end{equation}
By the universal property of the Clifford algebra, there exists a (uniquely defined) representation $(\tau, \mathcal H)$ of the \emph{complex} Clifford algebra $\mathbb Cl_n$, which satisfies $\tau(e_j)=E_j$, and conversely, any representation of the Clifford algebra is obtained in this manner.
Note that, for any $a\in \mathbb Cl_n$, $\tau(\overline a)$ is the adjoint of $\tau(a)$ for the Hilbert product on $\Sigma $.

Viewing $Spin_n$ as a subset of $\mathbb Cl_n$, we obtain by restriction of $\tau$ a representation of $Spin_n$ on $\mathcal H$, still denoted by $\tau$, which is unitary by the last remark.

The results concerning the \emph{irreducible} representations of $\mathbb Cl_n$ depend on the parity of $n$.

Assume first that $n$ is \emph{even}. Then the \emph{complex Clifford algebra} $\mathbb Cl_n$ has a unique irreducible representation (up to isomorphism) for $n$ even. Let $\Sigma_n$ the complex (finite-dimensional) Hilbert space on which the representation acts, and denote by $\sigma : \mathbb Cl_n \longrightarrow \End(\Sigma_n)$ the representation. 
 When restricted to $Spin_n$ (or equivalently to the even part $\mathbb Cl_n^{ev}$), the representation $(\sigma, \Sigma_n)$ splits into two inequivalent representations. 
 
In case $n$ is \emph{odd}, then there are two inequivalent irreducible representations, denoted by $(\sigma^+, \Sigma^+_n)$ and $(\sigma^-, \Sigma_n^-)$. To distinguish them, let
\[\omega^\mathbb C = i^{[\frac{n+1}{2}]}e_1e_2\dots e_n
\]
be the \emph{volume element}. Observe that the oddness of $n$ implies that  $\omega^\mathbb C$ is in the center of $\mathbb Cl_n$ . By Schur's lemma, $\tau(\omega)=\pm\Id$ for $\tau$ an irreducible representation. Hence $\sigma^\pm(\omega^\mathbb C)=\pm \Id$ on $\Sigma_n^\pm$, which distinguishes the two representations, and shows that they are not equivalent. When restricted to $\mathbb Cl_n^{ev}$ (or to $Spin_n$), both restrictions of $\sigma^\pm$ stay irreducible and are equivalent representations. Let $\Sigma_n=\Sigma_n^+\oplus \Sigma_n^-$, and let $\sigma = \sigma^+\oplus \sigma^-$.

In any case of parity, we call $(\sigma, \Sigma_n)$ \emph{the spinor representation} of $\mathbb Cl_n$ (or of $Spin_n$). On $\Sigma$, there is a Hermitian scalar product for which $\sigma(x)$ is unitary for any $x\in \mathbb R^n$ with unit length. For this inner product, $\Sigma$ is a Clifford module. For $1\leq j\leq n$, let $E_j=\sigma(e_j)$. Then the $E_j$'s are skew Hermitian and satisfy the defining  relations \eqref{Ej}.

Finally, we have to connect the standard realization of $Spin_n$ with the realization of $\widetilde M$ in the Vahlen-Maass-Ahlfors approach. The linear map $\gamma : E^{n-1}\longmapsto \mathbb Cl^{ev}_n$ given by
\[\gamma(e_j) = e_1e_j, \quad 2\leq j\leq n
\]
satisfies $\gamma(e_i)\gamma(e_j) +\gamma(e_j)\gamma(e_i)= -2\delta_{ij} $ and hence can be extended to yield an isomorphism (still denoted by $\gamma$) of $Cl(E^{n-1})$ onto $Cl_n^{ev}$. The map 
$\gamma$ induces an isomorphism of $\widetilde M$ onto $Spin_n$. For $a\in Cl(E^{n-1})$,  let $\tau(a) = \sigma(\gamma(a))$. Then $(\tau, \Sigma)$ is a representation of $Cl(E^{n-1})$ and, by restriction to $\widetilde M$,  a representation of $\widetilde M$, called the \emph{spinor representation}. For $x\in E^n$, 
\begin{equation}\tau(x_1+x_2e_2+\dots+x_ne_n) =  x_1\Id + x_2 E_1E_2+\dots + x_n E_1E_n\ .
\end{equation}

Recall that the element $\widetilde w = \begin{pmatrix} 0&-1\\1&0\end{pmatrix}$ is a representative of the nontrivial Weyl group element.
\begin{proposition} Let $\tau$ be the spinor representation of $\widetilde M$, and let $w\tau$ be the representation of $\widetilde M$ given by $w\tau(m) = \tau(\widetilde w^{-1}m\widetilde w)$. Then
\[w\tau=\tau'\ ,
\]
where $\tau'$ is the restriction to $\widetilde M$ of the representation of $\mathbb Cl_{n-1}$ given by $\tau'(a)=\tau(a')$, for $a\in \mathbb Cl_{n-1}$.
\end{proposition}

\begin{proof} For $a\in \Gamma_n$,

\[ \begin{pmatrix} 0&1\\-1&0\end{pmatrix} \begin{pmatrix} a&0\\0&{a^*}^{-1}\end{pmatrix}
 \begin{pmatrix} 0&-1\\1&0\end{pmatrix}= \begin{pmatrix}{a^*}^{-1}&0\\0&a\end{pmatrix}\ .
\]
If moreover, $\vert a \vert =1$, then ${a^*}^{-1}= a'$, so that the automorphism of $\widetilde M$ induced by $\widetilde w$ coincides with the principal automorphism. The statement follows.
\end{proof}

\begin{proposition}
 For any $a\in Cl(E^{n-1})$,
\begin{equation}\label{tautau'}
E_1\,\tau'(a) = \tau(a)\,E_1\ .
\end{equation}

\end{proposition}

\begin{proof} It suffices to verify the statement for any vector $x\in E^n$. Let $x= x_1+x_2e_2+\dots +x_ne_n$. Then
\[ \tau(x)E_1 = (x_1\Id + x_2 E_1E_2+\dots x_nE_1E_n)\,E_1 = x_1 E_1 +x_2E_2+\dots +x_nE_n\]
whereas
\[ E_1 \tau'(x) = E_1\,(x_1\Id-x_2E_1E_2-\dots -x_nE_1E_n) = x_1 E_1+x_2E_2+\dots +x_nE_n\ .\]
\end{proof}

\section{ The principal spinorial series of $\widetilde G$ and the associated intertwining operators}

Let us first recall the general theory of Knapp-Stein intertwining operators. Let $G$ be a semisimple Lie group (connected and with finite center), $P$ a minimal parabolic subgroup. Let $\theta$ be a Cartan involution, with fixed points $K$, which is a maximal compact subgroup of $G$.  Let $P=MAN$ be a Langlands decomposition of $P$ adapted to $\theta$. In particular, $M=P\cap K$. Let  $M'$ be the normalizer of $A$ and $W\simeq M'/M$ be the  corresponding Weyl group. Let $X=G/P\simeq K/M$, and let $o=eP$ be the origin in $X$.
Let $\mathfrak a$ be the Lie algebra of $A$, and let $\exp$ be the exponential map from $\mathfrak a$ onto $A$. Let $\rho\in \mathfrak a'$ (the dual of $\mathfrak a$) be the half-sum of the positive roots relative to $N$.

The map 
\[ K\times A\times N \ni (k,a,n)\longmapsto kan \in G\]
is a diffeomorphism onto $G$. If $g$ is in $G$, we write $g=\kappa(a) \exp H(g) \nu(g)$ for the \emph{Iwasawa decomposition} .

Let $\overline N = \theta N$. The map
\[\overline N\times M\times A\times N \ni (\overline n,m,a,n)\longmapsto \overline nman \in G\]
is a diffeomorphism onto a dense open set of $G$. For $g$ an element in the image, let
\[ g=\overline n(g) m(g)a(g)n(g)\]
be the corresponding \emph{Bruhat decomposition}.

Let $\tau$ be a unitary representation of $M$ on a (finite dimensional) Hilbert space $V$ and let $\lambda$ be a complex linear form on $\mathfrak a = Lie(A)$. Let $\tau_\lambda$ be the representation of $P$ defined by  
\[ \tau_\lambda(man) = a^\lambda \tau(m), \quad m\in M,\ a\in A,\ n\in N\ ,
\]
where $a^\lambda = e^{\lambda(\log a)}$.

Form the \emph{induced representation} $\pi_{\tau,\lambda}= \Ind_{MAN}^G (\sigma\otimes\exp \lambda\otimes 1)$.
We will work with the \emph{noncompact picture} of this induced  representation. Introduce the space $L^2_\lambda(\overline N)$ as the space of functions $f$ on $\overline N$, valued in $V$, which satisfy
\[ \int_{\overline N} \vert f(x)\vert^2 e^{2\Re \lambda \big(H(\overline n)\big)} d\overline n <+\infty\ .\]
 We state the noncompact realization of the induced representation as a proposition (see \cite{kn} ch VII).

\begin{proposition} For $g\in G$, 
\begin{equation}
\pi_{\tau,\lambda}(g) f (\overline n) = e^{-(\lambda+\rho)\log a(g^{-1}\overline n)}\tau\big(m(g^{-1} \overline n)\big)^{-1} f\big(\overline n(g^{-1}\overline n)\big)\ .
\end{equation}
defines a representation $\pi_{\tau,\lambda}$ of $G$ by bounded operators on $L^2_\lambda(\overline N)$.
\end{proposition}

Let $w$ be an element of $W$, and choose a representative (still denoted by $w$) in $M'$. Let $w\tau$ be the representation of $\widetilde M$ defined by $w\tau(m) = \tau(\widetilde w^{-1}m\widetilde w)$. The Knapp-Stein theory of intertwining operators offers a construction of an intertwining operator between $\pi_{\tau,\lambda}$ and $\pi_{w\tau,w\lambda}$. We state it as a proposition (see \cite{kn} ch VII, (7.39).

Set, for $f$ a function on $\overline N$ and $x\in \overline N$ 
\[J_{\tau,\lambda,w} f(x)= \int_{\overline N} e^{(-\rho+\lambda) \log a(w^{-1}\overline n)} \tau(m(w^{-1}\overline n)) f(x\overline n) d\overline n
\]

\begin{proposition} The operator $J_\lambda$ is an intertwining operator between $\pi_{\tau,\lambda}$ and $\pi_{w\tau,w\lambda}$, namely for any $g\in G$,
\[J_{\tau,\lambda,w}\circ \pi_{\tau, \lambda}(g) = \pi_{w\tau, w\tau_\lambda}(g)\circ J_{\tau,\lambda,w}\ .
\]
\end{proposition}

The proposition lets aside the convergence of the integral, which is true  for $\lambda$ in some open set of $\mathfrak a'_\mathbb C$. The intertwining operator $J_{\tau,\lambda,w}$ can then be extended meromorphically to the whole space.

This general scheme applies to our situation. Let $(\tau, \Sigma)$ be the spinor representation of $\widetilde M$ and let $\lambda\in \mathbb C$. Define the representation $\tau_\lambda$ of $\widetilde P$ by
\[\tau_\lambda(ma_tn) = t^{2\lambda} \tau(m)\ .
\]

Let $\pi_\lambda= {\Ind\, } _{\widetilde P}^{\widetilde G}\, \tau_\lambda$. Following the procedure just described above, we obtain the following realization of these representations (\emph{noncompact picture}).

\begin{theorem} For $\lambda\in \mathbb C$ and $g=\begin{pmatrix} a&b\\c&d\end{pmatrix}$, the formula
\[\pi_\lambda(g) f(x) = \vert d^*-b^*x\vert^{-2\lambda-n}\, {\tau\Bigg(\frac{d^*-b^*x}{\vert d^*-b^*x\vert}\Bigg)}^{-1}\,f\big((-c^*+a^*x)(d^*-b^*x)^{-1}\big)
\]
defines a representation of $\widetilde G$ by bounded operators on $L^2_\lambda (E^n,\Sigma)$.
\end{theorem}

\begin{proof} For $x\in E$,
 \[g^{-1}\overline n_x= \begin{pmatrix} d^*-b^*x&-b^*\\-c^* +a^*x&a^*\end{pmatrix} .\]
 The components in the Bruhat decomposition of this element are
 \[(ma)(g^{-1} \overline n_x) = d^*-b^*x, \quad \overline n (g^{-1} \overline n_x)= (-c^*+a^*x)(d^*-b^*x)^{-1}\ .
 \]
 Hence the formula is a consequence of the general statement.
\end{proof}

There is another closely related representation. In fact, the same construction can be done using the representation $\tau'$ of $\widetilde M$ we introduced earlier instead of $\tau$. The corresponding representation will be denoted by $\pi_\lambda'$. It is related to $\pi_\lambda$ be the following elementary result, which follows from Proposition \ref{tautau'}.

\begin{proposition} For any $g\in \widetilde G$
\begin{equation}\label{E1intw}
E_1\pi'_\lambda (g) = \pi_\lambda(g)E_1\ .
\end{equation}
\end{proposition}

We now apply Knapp-Stein theory to get an intertwining operator between $\pi_\lambda$ and $\pi_{-\lambda}'$.

Let $\overline n = \begin{pmatrix}1&0\\y&1\end{pmatrix}$. Then
\[ w^{-1}\overline n = \begin{pmatrix}0&1\\-1&0\end{pmatrix} \begin{pmatrix}1&0\\y&1\end{pmatrix}=\begin{pmatrix} y&1\\-1&0\end{pmatrix}
\]
so that \[t(w^{-1} \overline n) = \vert y\vert, \quad m(w^{-1}\overline n) =  \frac{y}{\vert y\vert}\ .\]
Hence the Knapp-Stein operator is given by

\[J_\lambda f(x) = \int_{E^n} \vert y\vert^{2\lambda-n} \tau\Big(\frac{y}{\vert y\vert}\Big)f(x-y) dy\]

\begin{theorem} For $\lambda\in \mathbb C$, let
\begin{equation}
J_\lambda f(x) = \int_{E^n} \vert y\vert^{2\lambda-n} \tau\Big(\frac{y}{\vert y\vert}\Big)f(x-y) dy\ .
\end{equation}
For $\Re \lambda >0$ and $f\in L^2_\lambda(E^n, \Sigma)$, the integral converges and the operator $J_\lambda$ thus defined is bounded on $L^2_\lambda(E^n, \Sigma)$ and for any $g\in \widetilde G$ satisfies
\begin{equation}\label{intw}
J_\lambda\, \pi_\lambda(g) =\pi_{-\lambda}'(g)\,  J_\lambda\ .
\end{equation}
\end{theorem}

Although this is a consequence of the Knapp-Stein theory, we may offer a direct proof (compare with \cite{gm}).

\begin{proof} Let $g=\begin{pmatrix}a&b\\c&d\end{pmatrix}$. Let $\gamma=\begin{pmatrix} a^*&-c^*\\-b^*&d^*\end{pmatrix}$ which is easily seen to be an element of $\widetilde G$.

Let $f$ be in $\mathcal C^\infty_c(E^n, \Sigma)$. Then 
\[J_\lambda f (x) = \int_{E^n} \vert x-y\vert^{2\lambda-n} \tau\Big(\frac{x-y}{\vert x-y\vert}\Big)f(y)\,dy\ ,
\]
so that
\[\pi'_{-\lambda} (g)J_\lambda f(x)  = \dots\]\[= \vert d^*-b^*x\vert^{2\lambda-n}\tau'\Big(\frac{d^*-b^*x}{\vert d^*-b^*x}\Big)^{-1}\int_{E^n}\vert \gamma(x)-y\vert^{2\lambda-n}\tau\Big(\frac{\gamma(x)-y}{\vert \gamma(x)-y\vert}\Big) f(y)dy\ .
\]
Use the change of variable $y=\gamma(z)$, and hence $dy = \vert d^*-b^*z\vert^{-2n} dz$ to get
\[\vert d^*-b^*x\vert^{-2\lambda-n} \tau'\Big(\frac{d^*-b^*x)}{\vert d^*-b^*x\vert}\Big)^{-1} \int_{E^n}\vert \gamma(x)-\gamma(z)\vert^{2\lambda-n} \tau\Big(\frac{\gamma(x)-\gamma(z)}{\vert \gamma(x)-\gamma(z)}\Big) \vert d^*-b^*z\vert^{-2n} f\big(\gamma(z)\big)dz\ .
\]
Now, using \eqref{globalcov}

\[\gamma(x)-\gamma(z) = {(d^*-b^*x)^*}^{-1}(x-z)(d^*-b^*z)^{-1}
\]
\[ \vert \gamma(x)-\gamma(z)\vert = \vert d^*-b^*x\vert^{-1} \vert x-z\vert \vert d^*z-b^*z\vert^{-1} \]
\[ \frac{\gamma(x)-\gamma(z)}{\vert\gamma(x)-\gamma(z) \vert} ={\Big( \frac{d^*-b^*x)}{\vert d^*-b^*x\vert}\Big)^*}^{-1} \frac{x-z}{\vert x-z\vert}\Big(\frac{d^*-b^*z}{\vert d^*-b^*z}\Big)^{-1}\ .
\]
For $u\in \Gamma_n$ such that $\vert u \vert =1$, ${\tau(u^*}^{-1})= \tau(u') = \tau'(u)$ so that
\[\pi'_{-\lambda}(g)J_\lambda f(x) \]\[= \int_{E^n} \vert x-z\vert^{2\lambda-n}\vert d^*-b^*z\vert^{-2\lambda-n} \tau\Big(\frac{d^*-b^*z}{\vert d^*-b^*z\vert}\Big)^{-1}f\big((a^*z-c^*)(d^*-b^*z)^{-1}\big) dz
\]
\[ = J_\lambda \pi_\lambda(g)f(x)\ .
\]
\end{proof}

 For $s\in \mathbb C$, let, for $x\in E^n, x\neq 0$
 \[d_s(x) = \vert x\vert^{s-1} \sum_{j=1}^n x_jE_j\ ,
 \]
 and let $D_s$ be the associated convolution operator  defined by
 \[D_sf(x) = \int_{E^n} d_s(y) f(x-y) dy\ .
 \]
 
 \begin{proposition} Let $\Re s>-n$. For any $g\in \widetilde G$,
 \begin{equation}\label{dsintw}
 D_s\circ \pi'_{\frac{s+n}{2}}(g) = \pi'_{-(\frac{s+n}{2})}(g)\circ D_s
 \end{equation}
 
 \end{proposition}
 
 \begin{proof}
 By using the trivial intertwining relations \eqref{tautau'} and \eqref{E1intw}, we can transform the intertwining relation \eqref{intw} to get

\[ J_\lambda E_1 \circ \pi'_{\lambda}(g) = \pi'_{-\lambda}(g)\circ J_\lambda E_1\ ,\]
for any $g\in \widetilde G$. Next, for any $x\in E^n, x\neq 0$
 \[\vert x\vert^{2\lambda-n-1}(x_1+x_2E_1E_2+\dots + x_nE_1E_n)E_1 = d_{2\lambda-n}(x)\ .\] 
 The statement follows, with $\lambda = \frac{s+n}{2}$.

 \end{proof}

\section {The fundamental identity and its consequences}

For $\Re s>-n$, $d_s$ is an integrable function on $E^n$ (with values in $\End(\Sigma)$), hence a (tempered) distribution. We want to meromorphically continue this distribution to $\mathbb C$.

Let $D$ be the Dirac operator on $E^n$. By definition, it acts on smooth functions on $E^n$ with values in $\Sigma$ by
\[ Df(x) = \sum_{j=1}^n E_j\frac{\partial f}{\partial x_j} (x)\ .\]

Extend this formula to $\End(\Sigma)$ valued function : if $S(x)$ is such a function, let $\displaystyle DR(x) = \sum_{j=1}^n E_j \frac{\partial R}{\partial x_j}(x)$. Notice that the associated convolution operator satisfies
\[D(R\star f) = DR\star f\ .
\]
In both cases,  $D^2 f= -\sum_{j=1}^n \frac{\partial^2f}{\partial x_j^2}$, which we write as $D^2 = \Delta$, where $\Delta$ is the (extension to $\Sigma$ or $End(\Sigma)$-valued functions of the) standard Laplacian on $E^n$.
\begin{proposition} [Fundamental identity] Let $s\in \mathbb C, \Re s > -n$. Then, for $x\in E^n, x\neq 0$
\begin{equation}\label{PBS}
D \vert x\vert^{s+1} = (s+1) d_s\ ,
\end{equation}
where both sides are $End(\Sigma)$-valued functions.
\end{proposition}

\begin{proof} It amounts to the formula 
\[\frac{\partial}{\partial x_j} \vert x\vert^{s+1} =(s+1)\,x_j\,\vert x \vert^{s-1}\ .
\]
\end{proof}

As the meromorphic continuation of the distribution $\vert x\vert^s$ is well known (see \cite{gs}), equation \eqref{PBS} allows the meromorphic continuation of the distribution $d_s$.

\begin{proposition} The distribution $\vert x\vert^s$ can be continued meromorphically to $\mathbb C$, with simple poles at $s=-n-2k$, for $k\in \mathbb N$. The residue at $s=-n-2k$ is given by
\[Res\,(\vert x\vert^s, -n-2k) = c_k \Delta^k \delta\ ,
\]
where $\displaystyle c_k= \frac{2\pi^{\frac{n}{2}}}{\Gamma(\frac{n}{2})}\, \frac{1}{2^k \,k!\,n(n+2)\dots(n+2k-2)}\ $.

\end{proposition}

\begin{proposition} The distribution $d_s$ can be meromorphically continued to $\mathbb C$ with simple poles at $s=-n-2k-1, k\in \mathbb N$. The residue of $d_s$ at $-n-2k-1$ is given by
\[Res\,(d_s, -n-1-2k) = c_k \, D^{2k+1}\delta\ .
\]

\end{proposition}

\begin{proof} Let $f$ be a function in $\mathcal C^\infty_c(E^n,\Sigma)$. Use \eqref{PBS}, recall that the $E_j$'s are skew Hermitian and integrate by part to get

\[\int_{E^n} d_s(x) f(x) dx = \frac{1}{s+1} \int_{E^n} \vert x\vert^{s+1} Df(x)dx\ .
\]
This identity is valid {\it a priori} for $\Re s>-n$. The right hand side can be extended to a meromorphic function, with poles at $s+1 = -n-2k, k\in \mathbb N$. This serves to \emph{define} the left hand side. At $s=-n-2k-1$, the residue of the right hand side is 
$c_k\, \Delta^k(Df)(0)$. But $D^2 = \Delta$, so that $\Delta^k = D^{2k}$,  hence the proposition follows.

\end{proof}

\begin{theorem} For any positive integer $k$, and any $g\in \widetilde G$,
\begin{equation}\label{Dintw}
 D^{2k+1} \circ \pi'_{-k-\frac{1}{2}}(g) = \pi_{k+\frac{1}{2}}'(g)\circ D^{2k+1}
\end{equation}
\end{theorem}
\begin{proof} Recall the intertwining relation \eqref{dsintw}. It is clearly valid for $s$ in $\mathbb C$, provided $s$ is not a pole. But at a pole, say $s=-n-2k-1$, we may pass to the limit on both sides of \eqref{intw}, thus obtaining \eqref{Dintw}.

\end{proof}

\section { The compact picture}

The results above may also be realized in the compact picture, i.e. on the sphere, where the induced representation,
the Knapp-Stein operators and their residues may also be computed. The main result is
an explicit expression of the residues as a polynomial in the Dirac operator of the sphere. 
Recall from \cite{kn} that our induced representation may also be realized in the compact picture, see Chapter VII,
(7.3a), and the intertwining operator given as in \cite{kn} (7.37) as follows:
Set, for $f$ a function on $K$ and $x\in K$
\[J_{\tau,\lambda,w} f(x)= \int_K e^{(-\rho+\lambda) \log a(w^{-1}k)} \tau(m(w^{-1}k)) f(xk) dk 
\ .\]

Now we can repeat the arguments from Euclidian space and realize the Knapp-Stein operator as
a kernel operator, acting on sections of the spin bundle over $S^n$, and we may find the
residues of this meromorphic family. For this is it convenient to calculate the spectrum of
the Knapp-Stein operator, i.e. its eigenvalues on the $K$-types in the induced representation.
Recall the method of spectrum-generating \cite{boo}, which we can apply in an elementary way to obtain
the $K$-spectrum as in the following result.

When $n$ is odd, the spin representation of $\widetilde M$ has highest weight $(\frac{1}{2}, \dots , \frac{1}{2})$ (we use the standard choices of a Cartan subalgebra of $\mathfrak m$ and a basis inside). The induced representation space (sections of the spin bundle over $S^n$) decomposes under the action of $\widetilde K$ without multiplicity, and the corresponding highest weights of the $\widetilde K$-types are $(j,\pm) =(j, \frac{1}{2}, \dots ,\frac{1}{2}, \pm \frac{1}{2})$ with $j \in \mathbb N + \frac{1}{2}$. 

When $n$ is even, the spin representation of $\widetilde M$ is a sum of two representations, say $\sigma^+$ and $\sigma^-$, with respective highest weights $(\frac{1}{2},\dots, \frac{1}{2}, \frac{1}{2})$ and $(\frac{1}{2},\dots, \frac{1}{2}, -\frac{1}{2})$. Each induced representation ($(\pi^+_\lambda, \mathcal S^+_\lambda)$ from $\sigma_+$, $(\pi^-_\lambda, \mathcal S^-_\lambda)$  from $\sigma_-$) decomposes under the action of $\widetilde K$ without multiplicity, and the corresponding heighest weights of the $\widetilde K$-types are $(j, \frac{1}{2},\dots,\frac{1}{2})$ with $j\in \mathbb N +\frac{1}{2}$.

\begin{proposition} Define the spectral functions as in \cite{boo} by
$$Z_{j,\pm}(\lambda) = \pm \frac{\Gamma(\frac{n}{2} + j -\lambda)}{\Gamma(\frac{n}{2} + j +\lambda)}$$
When $n$ is odd, the operator acting on the $(j, \pm)$ $\widetilde K$-type by the scalar $Z_{j,\pm}(\lambda)$
is an intertwining operator between $\pi_{\lambda}$ and $\pi_{-\lambda}$.

\noindent
When $n$ is even, the operator acting from $\mathcal S^\pm_\lambda$ into $\mathcal S^\mp_\lambda$ on the $\widetilde K$-type $(j, \frac{1}{2},\dots, \frac{1}{2})$ by the scalar $Z_{j,\pm}(\lambda)$ is an intertwining operator between $\pi_\lambda= \pi_\lambda^+ \oplus \pi_\lambda^-$
 and $\pi_{-\lambda} = \pi_{-\lambda}^+ \oplus \pi_{-\lambda}^-$.
\end{proposition} 

See \cite{boo}, noticing that the parameter $r$ there coincides with $-\lambda$ in our present context. Because of the generic uniqueness of the intertwining operator between $\pi_{\lambda}$ and $\pi_{-\lambda}$, the  intertwining operator thus constructed is a multiple (by some meromorphic function of $\lambda$) of the one we use in the first part. The poles of the former were at $\lambda = -\frac{1}{2} -k, k\in \mathbb Z$. They now correspond to non singular values of the spectral functions, so that the residues are replaced by true values. The normalization is in fact such that the value of the intertwining operator at $\lambda = -\frac{1}{2}$ is exactly the Dirac operator on $S^n$. More precisely, for $\lambda=-\frac{1}{2}$, we obtain the spectrum of the Dirac operator on $S^n$.

\begin{proposition} The spectrum of the Dirac operator $\mathbb D$ is given by
\[Z_{k+\frac{1}{2},\, \pm} (-\frac{1}{2})= \pm (\frac{n}{2}+k)\ .
\]

\end{proposition}

\noindent
{\bf Remark}. An alternative determination of the spectrum of the Dirac operator on $S^n$ was given in \cite{bo}, using a more complicated argument which however is only using conformal geometry). See also \cite{ch}.
\medskip

For the other poles $\lambda = -\frac{1}{2}-m, m\in \mathbb N$, the computation of the values of the spectral functions and the previous result yield the following result.
\begin{theorem} Let $m\in \mathbb N$. The differential operator

\[\mathbb D_m = \mathbb D(\mathbb D^2-1)(\mathbb D^2-4)\dots (\mathbb D^2-m^2)
\]
is covariant with respect to $(\pi_{-\frac{1}{2}-m},\pi_{\frac{1}{2}+m})$. 

\end{theorem}
\begin{proof} For $j= \frac{1}{2}+ k$,
\[Z_{k+\frac{1}{2},\, \pm}(-\frac{1}{2} -m)=\pm \,(\frac{n}{2} +k+m)(\frac{n}{2} +m-1)\dots(\frac{n}{2}+k-m)
\]
which coincides with the spectral function of the operator \[\mathbb D_m =(\mathbb D+m)(\mathbb D+m-1)\dots (\mathbb D-m)\ .\] Hence the statement.
\end{proof}
 
\medskip 

For other approaches to the covariance properties of powers of the Dirac operator on $S^n$, see \cite{es}, \cite{lr}.

\medskip
\footnotesize{\noindent Address\\ Jean-Louis Clerc, Institut Elie Cartan, Universit\'e de Lorraine, 54506 Vand\oe uvre-l\`es-Nancy, France\\ Bent \O rsted, Matematisk Institut, Byg.\,430, Ny Munkegade, 8000 Aarhus C,
Denmark.\\}
\medskip

\noindent \texttt{{jean-louis.clerc@univ-lorraine.fr, orsted@imf.au.dk
}}

\end{document}